\documentclass[10pt]{amsart}
\usepackage{amssymb}
\usepackage{amsmath}
\usepackage{amsfonts}
\usepackage{palatino}

\newcommand{\Pic}{{\rm Pic}}

\newcommand{\gon}{{\rm gon}}
\newcommand{\Cliff}{{\rm Cliff}}

\renewcommand{\det}{{\rm det}}

\renewcommand{\dim}{{\rm dim}}

\renewcommand{\det}{{\rm det}}

\theoremstyle{plain}

\newtheorem{thm}{Theorem}[section]

\newtheorem{cor}[thm]{Corollary}
\newtheorem{prop}[thm]{Proposition}
\newtheorem{lem}[thm]{Lemma}

\newtheorem{con}[thm]{Conjecture}

\theoremstyle{definition}

\newtheorem{rmk}[thm]{Remark}

\def\PP{{\textbf P}}
\def\OO{\mathcal{O}}

\def\P{\mathcal{P}}
\def\W{\mathcal{W}}

\def\G{\mathcal{G}}

\def\I{\mathcal{I}}

\def\cM{\mathcal{M}}

\def\mm{\overline{\mathcal{M}}}

\def\Hom{\mathrm{Hom}}
\def\Ext{\mathrm{Ext}}

\begin{document}
\title{Green's Conjecture for curves on arbitrary $K3$ surfaces}
\author[M. Aprodu]{Marian Aprodu}
 \email{aprodu@imar.ro}
\address{Institute of Mathematics "Simion Stoilow" of the Romanian
Academy, RO-014700 Bucharest, Romania,
and \c Scoala Normal\u a Superioar\u a Bucure\c sti,
Calea Grivi\c tei 21, Sector 1, RO-010702 Bucharest, Romania}
\author[G. Farkas]{Gavril Farkas}
\email{farkas@math.hu-berlin.de}
\address{Humboldt-Universit\"at zu Berlin, Institut F\"ur Mathematik,
10099 Berlin, Germany}


\thanks{Research of the first author partly supported by a
PN-II-ID-PCE-2008-2 grant (cod 1189, contract no 530)
and by a resumption of a Humboldt Research Fellowship. Research of both authors
partly supported by the
Sonderforschungsbereich "Raum-Zeit-Materie".
MA thanks HU Berlin for the kind hospitality and the excellent
atmosphere during the preparation of this work.}

\begin{abstract}
 Green's Conjecture predicts than one can read off special linear series on an algebraic curve,
by looking at the syzygies of its canonical embedding. We extend Voisin's results on syzygies of
$K3$ sections, to the case of $K3$ surfaces with arbitrary Picard lattice. This, coupled with results
of Voisin and Hirschowitz-Ramanan, provides a complete solution to Green's Conjecture for
smooth curves on arbitrary $K3$ surfaces.
\end{abstract}

\maketitle

\section{Introduction}
\label{sec: intro}

Green's Conjecture on syzygies of canonical curves asserts that one can recognize existence of special linear series on an algebraic curve,
by looking at the syzygies of its canonical embedding. Precisely, if $C$ is a smooth algebraic curve of genus $g$, $K_{i, j}(C, K_C)$ denotes the $(i, j)$-th Koszul cohomology group of the canonical bundle $K_C$ and $\mbox{Cliff}(C)$ is the Clifford index of $C$, then M. Green \cite{Gr84} predicted the vanishing statement
\begin{equation}\label{greenconj}
K_{p, 2}(C, K_C)=0, \ \mbox{ for all } \ p<\mathrm{Cliff}(C).
\end{equation}
In recent years, Voisin \cite{Voisin: even}, \cite{Voisin: odd} achieved a major breakthrough by
showing that Green's Conjecture holds for smooth
curves $C$ lying on $K3$ surfaces $S$ with $\mbox{Pic}(S)=\mathbb Z\cdot C$. In particular, this
establishes Green's Conjecture for general curves of every genus. Using Voisin's work, as well as a
degenerate form of \cite{HR98}, it has been proved in \cite{Ap05} that Green's Conjecture holds
for any curve $C$ of genus $g$ of gonality $\mbox{gon}(C)=k\leq(g+2)/2$, which satisfies the \emph{linear
growth condition}
\begin{equation}\label{lgc}
\mbox{dim } W^1_{k+n}(C)\leq n, \mbox{ for } 0\leq n\leq g-2k+2.
\end{equation}
Thus Green's Conjecture becomes a question in Brill-Noether theory. In particular, one can
check that condition (\ref{lgc}) holds for a general curve $[C]\in \cM_{g, k}^1$ in any gonality stratum of $
\cM_g$, for all $2\leq k\leq (g+2)/2$. Our main result is the following:

\begin{thm}\label{maintheorem}
Let $S$ be a K3 surface and $C\subset S$ be a smooth curve with
$g(C)=g$ and $\gon(C)=k$. If $k\le (g+2)/2$, then $C$ satisfies Green's conjecture.
\end{thm}

Note that Theorem \ref{maintheorem} has been established in \cite{Voisin: even}
when $\Pic(S)=\mathbb{Z}\cdot C$.
The proof relies (via \cite{Ap05}) on the case of curves of odd genus
of maximal gonality. Precisely, when $g(C)=2k-3$ and $\mbox{gon}(C)=k$,
Green's conjecture is due to Voisin \cite{Voisin: odd} combined with results of Hirschowitz-Ramanan \cite{HR98}. Putting together these results and Theorem \ref{maintheorem},
we conclude:

\begin{thm}
Green's Conjecture holds for every smooth curve $C$ lying on an arbitrary $K3$ surface $S$.
\end{thm}

In the proof of Theorem \ref{maintheorem}, we distinguish two cases. When $\mbox{Cliff}(C)$ is
computed by a pencil (that is, $\mbox{Cliff}(C)=\mbox{gon}(C)-2$), we use a parameter count for
spaces of Lazarsfeld-Mukai bundles \cite{La86}, \cite{CP95}, in order to find a smooth
curve $C'\in |C|$, such that $C'$ verifies condition (\ref{lgc}).
Since Koszul cohomology satisfies the Lefschetz hyperplane principle, one has that $K_{p, 2}(C,
K_C)\cong K_{p, 2}(C', K_{C'})$. This proves Green's Conjecture for $C$.

When $\mbox{Cliff}(C)$ is no longer computed by a pencil, it follows from
\cite{CP95}, \cite{Kn09} that either $C$ is
a smooth plane curve or else, a \emph{generalized ELMS example}, in the sense that there exist smooth curves $D, \Gamma
\subset S$, with $\Gamma^2=-2, \Gamma\cdot D=1$ and $D^2\geq 2$, such that $C\equiv 2D+
\Gamma$ and $\mbox{Cliff}(C)=\mbox{Cliff}(\OO_C(D))=\mbox{gon}(C)-3$. This case requires a
separate analysis, similarly to
\cite{ApP08}, since condition (\ref{lgc}) is no longer satisfied, and we refer to Section \ref{sec: ELMS} for details.

Theorem \ref{maintheorem} follows by combining results obtained by using the powerful techniques developed in \cite{Voisin: even}, \cite{Voisin: odd}, with facts about the
effective cone of divisors of $\mm_g$. As pointed out in \cite{Ap05}, starting from a $k$-gonal
smooth curve $[C]\in \cM_{g}$ satisfying the Brill-Noether growth condition (\ref{lgc}), by identifying pairs of general points $x_i,
y_i\in C$ for $i=1, \ldots, g+3-2k$ one creates a \emph{stable} curve
$$\bigl[X:=C/x_1\sim y_1, \ldots,   x_{g+3-2k}\sim y_{g+3-2k} \bigr]\in \mm_{2g+3-2k}$$
having maximal gonality $g+3-k$, that is, lying outside the closure of the Hurwitz divisor $\cM_{2g+3-2k, g+3-k}^1$ consisting of curves with a pencil $\mathfrak g^1_{g+3-k}$. Since the class of the virtual failure locus of Green's Conjecture
is a multiple of the Hurwitz divisor $\mm_{2g+3-2k, g+3-k}^1$ on $\mm_{2g+3-2k}$, see \cite{HR98},  Voisin's theorem can be
extended to all irreducible stable curves of genus $2g+3-2k$ and having maximal gonality, in particular to $X$ as well, and a
posteriori to smooth curves of genus $g$ sitting on $K3$ surfaces with arbitrary Picard lattice. On the other
hand, showing that condition (\ref{lgc}) is satisfied for a curve $[C]\in \cM_g$, is a question of pure Brill-Noether nature.

\vskip 6pt

Theorem \ref{maintheorem} has strong consequences on Koszul cohomology of $K3$ surfaces. It is known that for any globally generated line bundle $L$ on a $K3$ surface $S$,
the Clifford index of any smooth irreducible curve is constant, equal to, say
$c$, \cite{GL87}.
Applying Theorem \ref{maintheorem}, Green's hyperplane
section theorem, the duality theorem  and finally the Green-Lazarsfeld nonvanishing theorem \cite{Gr84},
we obtain a complete description of the distribution of zeros among the
Koszul cohomology groups of $S$ with values in $L$.

\begin{thm}
\label{thm: K3}
Suppose $L^2=2g-2\ge 2$. The Koszul cohomology group $K_{p,q}(S,L)$ is nonzero if and only if one of the following cases occur:
\begin{enumerate}
\item $q=0$ and $p=0$, or
\item $q=1$, $1\le p\le g-c-2$, or
\item $q=2$ and $c\le p\le g-1$, or
\item $q=3$ and $p=g-2$.
\end{enumerate}
\end{thm}

The analysis of the Brill-Noether loci implies also  that
the Green-Lazarsfeld Gonality Conjecture is satisfied for
curves of Clifford dimension one on arbitrary K3 surfaces,
general in their linear systems, see Section \ref{sec: GL} for
details.

\section{Brill-Noether loci and their
dimensions}
\label{sec: dim BN}

Throughout this section we fix a $K3$ surface $S$ and a globally generated
line bundle $L\in \mbox{Pic}(S)$. We recall \cite{SD}, that the assumption that $L$ be globally generated is equivalent to $|L|$ having no base components. We denote by $|L|_s$ the locus of smooth
connected curves in $|L|$. For integers $r, d\geq 1$,
we consider the morphism
$\pi_S:\W^r_d(|L|)\to |L|_s$ with fibre
over a point $C\in |L|_s$ isomorphic to the Brill-Noether locus
$W^r_d(C)$. The analysis of the Brill-Noether
loci $W^r_d(C)$ for a {\em general curve} $C\in |L|$ in its linear system,
is equivalent to the analysis of the restricted
maps $\pi_S:\W\to |L|$ over irreducible components $\W$
of $\W^r_d(|L|)$  dominating the linear system.
The main ingredient used to study $\W^r_d(|L|)$ is
the {\em Lazarsfeld-Mukai bundle} \cite{La86} associated to a complete
linear series.
To any pair
$(C,A)$ consisting of a curve $C\in |L|_s$  and a base point free linear series $A\in W^r_d(C)\setminus W^{r+1}_d(C)$, one associates
the {\em Lazarsfeld-Mukai bundle} $E_{C, A}:=F_{C, A}^{\vee}$ on $S$, via an elementary transformation along $C\subset S$:
\begin{equation}
\label{eqn: F}
 0\to F_{C,A}\to H^0(C,A)\otimes \mathcal O_S \buildrel{\mathrm{ev}}\over{\to} A\to 0.
\end{equation}
Dualizing the sequence (\ref{eqn: F}), we obtain the short exact sequence
\begin{equation}
\label{eqn: E}
 0\to H^0(C,A)^\vee\otimes \mathcal O_S \to E_{C, A} \to K_C\otimes A^{\vee}\to 0.
\end{equation}

The bundle $E_{C,A}$ comes equipped with a distinguished subspace of sections $H^0(C, A)^{\vee}\in G(r+1, H^0(S, E_{C, A}))$.
We summarize some characteristics of $E_{C,A}$:

\begin{prop}
\label{lem: E_{C,A}}
One has that
\begin{enumerate}
 \item $\det (E_{C,A})= L$.
 \item $c_2(E_{C,A})=d$.
 \item $h^0 (S,E_{C,A})= h^0(C,A)+h^1(C,A)$,\  \ $h^1(S,E_{C,A})=h^2(S,E_{C,A})=0$.
 \item $\chi(S,E_{C,A}\otimes F_{C,A})=
 2(1-\rho(g,r,d))$.
 \item $E_{C,A}$ is globally generated off the base locus of $K_C\otimes A^{\vee}$.
\end{enumerate}
\end{prop}
In particular, $E_{C,A}$ is globally generated if $K_C\otimes A^{\vee}$
is globally generated. Conversely, if $E$ is a globally
generated  bundle on $S$ with  $\mbox{rk}(E)=r+1$ and
$\det (E)= L$, there is a rational map
$
 h_E : G(r+1,H^0(S,E)) \dashrightarrow |L|.
$ defined in the following way.
A general subspace $\Lambda\in G(r+1,H^0(S,E))$
is mapped to the degeneracy locus of the evaluation map:
$
\mathrm{ev} _{\Lambda} : \Lambda \otimes \mathcal O_S \to E;
$
note that, generically, this degeneracy locus cannot be the whole
surface.
The image $h_E(\Lambda)$ is a smooth curve $C_{\Lambda}\in |L|$, and we set
$\mbox{Coker}(\mbox{ev}_{\Lambda}):=K_{C_{\Lambda}}\otimes A_{\Lambda}^{\vee}$, where $A_{\Lambda}\in \mbox{Pic}(C_{\Lambda})$ and $\mbox{deg}(A_{\Lambda})=c_2(E)$.

\begin{rmk}
\label{rmk: LM open}
A rank-$(r+1)$ vector bundle $E$ on $S$ is a Lazarsfeld-Mukai bundle if
and only if $H^1(S,E)=H^2(S,E)=0$ and  there exists an $(r+1)$-dimensional subspace
of sections $\Lambda\subset H^0(S,E)$, such that the the degeneracy locus of
the morphism $\mbox{ev}_\Lambda$ is a smooth curve. In particular, being a  Lazarsfeld-Mukai vector bundle is an open condition.
\end{rmk}

Coming back to the original situation when $C\in |L|_s$ and
$A\in W^r_d(C)\setminus W^{r+1}_d(C)$ is globally generated, we consider the Petri map
$$
 \mu_{0,A}: H^0(C, A) \otimes H^0(C, K_C\otimes A^{\vee})
 \to  H^0(C, K_C),
$$
whose kernel can be described in terms of Lazarsfeld-Mukai bundles.
Let $M_A$ the vector bundle of rank $r$
on $C$ defined as the kernel of the evaluation map
\begin{equation}\label{MA}
 0\to M_A \to  H^0(C,A)\otimes \mathcal O_C \buildrel{\mathrm{ev}}\over{\to} A\to 0.
\end{equation}
Twisting (\ref{MA}) with $K_C\otimes A^{\vee}$, we obtain that
$
 \mbox{Ker}(\mu_{0,A})= H^0(C, M_A\otimes K_C\otimes A^{\vee}).
$
Note also that there is an exact sequence sequence on $C$
\[
0\to \OO_C\to F_{C,A}\otimes K_C\otimes A^{\vee}\to M_A\otimes K_C\otimes A^{\vee}\to 0,
\]
while from the  defining sequence of $E_{C,A}$ one obtains the exact sequence on $S$
\[
0\to H^0(C, A)^\vee\otimes F_{C,A}\to E_{C,A}\otimes F_{C,A}\to
F_{C, A}\otimes K_C\otimes A^{\vee}\to 0.
\]

Since $h^0(C,F_{C,A})=h^1(C, F_{C, A})=0$, one writes that
\begin{equation}
\label{eqn: h^0(E F)}
H^0(C,E_{C,A}\otimes F_{C,A})=H^0(C,F_{C,A}\otimes K_C\otimes A^{\vee}).
\end{equation}

We shall use the following deformation-theoretic result \cite{Pare95}, which is a consequence of Sard's theorem applied to the projection
$\pi_S: \W^r_d(|L|)\rightarrow |L|$.

\begin{lem}
Suppose $\W\subset \W^r_d(|L|)$ is a dominating component, and
$(C,A)\in\W$ is a general element such that $A$ is
globally generated and $h^0(C,A)=r+1$. Then the
coboundary map
$
H^0(C,M_A\otimes K_C\otimes A^{\vee})\to H^1(C,\OO_C)$
is zero.
\end{lem}

The above analysis can be summarized as follows
(compare with \cite{ApP08}, Corollary 3.3):

\begin{prop}
\label{prop: dim bounds}
If $\W\subset \W^r_d(|L|)$ is a dominating component, and
$(C,A)\in\W$ is a general element such that $A$ is
globally generated and $h^0(C,A)=r+1$, then $\mathrm{dim }_A W^r_d(C)\leq  \rho(g,r,d)+
h^0(C,E_{C,A}\otimes F_{C,A})-1$.
Moreover, equality holds if $\W$ is reduced at $(C, A)$.
\end{prop}

In particular, if
$E_{C,A}$ is a simple bundle, then $\mu_{0, A}$ is injective and $\W$ is reduced at $(C, A)$ of dimension $\rho(g, r, d)+g$.
Thus, the problem of estimating $\mathrm{dim }_A  W^r_d(C)$, when $(C, A)\in \W$ is suitably general,
can be reduced to the case when $E_{C, A}$
is {\em not} a simple bundle.

\section{Varieties of pencils on $K3$ sections}
\label{sec: BN}

Throughout the remaining sections we mix the additive and the multiplicative notation for divisors and line bundles. If $L$ is a line bundle on a smooth projective variety $X$ and $L\in \mbox{Pic}(X)$ is a line bundle, we write $L\geq 0$ when $H^0(X, L)\neq 0$. If $E$ is a vector bundle on $X$ and $L\in \mbox{Pic}(X)$, we set $E(-L):=E\otimes L^{\vee}$.

As in the previous section, we fix a $K3$ surface $S$ together with a globally generated line bundle $L$ on $S$. We denote by $k$  the  gonality of a general smooth curve in the
linear system $|L|$, and set $g:=1+L^2/2$. Suppose that
$\rho(g, 1, k)\le 0$ (this leaves out one single case, namely $g=2k-3$, when $\rho(g, 1, k)=1$). Our aim is to prove the Koszul vanishing statement
$$K_{g-\mathrm{Cliff}(C)-1, 1}(C, K_C)=0,$$
for any curve $C\in |L|_s$. By duality, this is equivalent to Green's Conjecture for $C$.

It was proved in \cite{Ap05} that any smooth curve $C$ that satisfies the linear growth condition (\ref{lgc}), verifies both Green's and Green-Lazarsfeld Gonality Conjecture. By comments made in the previous section, a general curve $C\in |L|_s$
satisfies (\ref{lgc}), if and only if for any $n=0,\ldots,g-2k+2$,
and any irreducible component $W\subset W^1_{k+n}(C)$
such that a general element $A\in W$ is globally generated, has $h^0(C, A)=2$, and the corresponding Lazarsfeld-Mukai bundle
$E_{C,A}$ is not simple, the estimate
$
\dim \ W \le n$, holds.

Condition (\ref{lgc}) for curves which are general in their linear system, can be verified
 either by applying
Proposition \ref{prop: dim bounds}, or by estimating directly the dimension of the
corresponding irreducible components of the scheme
$\W^1_{k+n}(|L|)$.
In our analysis, we need the following description \cite{DM89} of non-simple
Lazarsfeld-Mukai bundles, see also \cite{CP95} Lemma 2.1:

\begin{lem}
\label{lemma: DM}
Let $E_{C,A}$ be a non-simple Lazarsfeld-Mukai bundle.
Then there exist line bundles $M,N\in \Pic(S)$ such that  $h^0(S,M),\ h^0(S,N)\ge 2$,
$N$ is globally generated, and there exists a zero-dimensional, locally complete
intersection subscheme $\xi$ of $S$ such that $E_{C,A}$ is
expressed as an extension
\begin{equation}
\label{eq: DM}
0\to M\to E_{C,A}\to N\otimes I_\xi \to 0.
\end{equation}
Moreover, if $h^0(S,M\otimes N^{\vee})=0$, then $\xi=\emptyset$ and the extension splits.
\end{lem}

We say that (\ref{eq: DM}) is the \emph{Donagi-Morrison} (DM) extension
associated to $E_{C, A}$.

\begin{lem}[compare with \cite{ApP08}, Lemma 3.6]
\label{lemma: DM unicity}
For any indecomposable non-simple Lazarsfeld-Mukai bundle $E$ on $S$,
the DM extension (\ref{eq: DM}) is uniquely determined by
$E$.
\end{lem}

\proof
We assume that two DM extensions
\[
0\to M_j\to E\to N_j\otimes I_{\xi_j}\to 0,\ \ j=1, 2,
\]
are given. Observe first
that $H^0(S,N_1\otimes M_2^{\vee})=H^0(S,N_2\otimes M_1^{\vee})=0$. Indeed, if
$N_1-M_2\ge 0$, we use $M_1-N_1\ge 0$, $M_2-N_2\ge 0$
(we are in the non-split case),
and $M_1+N_1=M_2+N_2=L$ to get a contradiction.
Then $H^0(S,(N_1\otimes M_2^{\vee})\otimes I_{\xi_1})=H^0(S,(N_2\otimes M_1^{\vee})\otimes
I_{\xi_2})=0$, so we obtain non-zero maps
$M_1\to M_2$ and $M_2\to M_1$. This implies that
$M_1=M_2$.
\endproof

\begin{rmk}
Similarly, one can prove that a
{\em decomposable} Lazarsfeld-Mukai bundle $E$ cannot
be expressed as an extension (\ref{eq: DM}) with $\xi\ne\emptyset$. Thus a DM extension is always
unique, up to a permutation of factors in the decomposable case.
Moreover, $E$ is decomposable
if and only if the corresponding DM extension is trivial.
\end{rmk}

The size of the space of endomorphisms of a non-simple Lazarsfeld-Mukai bundle can be explicitly computed
from the corresponding DM extension:

\begin{lem}
\label{lemma: endo}
Let $E$ be a non-simple Lazarfeld-Mukai bundle on $S$ with $\mathrm{det}(E)=L$,
and $M$ and $N$ the corresponding line bundles from the DM
extension. If $E$ is indecomposable, then
$$
h^0(S,E\otimes E^\vee)=1+h^0(S,M\otimes N^{\vee}).
$$
If $E=M\oplus N$, then
$
h^0(S,E\otimes E^\vee)=2+h^0(S, M\otimes N^{\vee})+ h^0(S, N\otimes M^{\vee}).
$
\end{lem}

\proof
The decomposable case being clear, we treat
the indecomposable case. Twisting the DM extension by $E^\vee$ and taking cohomology, we obtain the exact sequence
\[
0\to H^0(S,E^\vee(M))\to H^0(S,E\otimes E^\vee)\to
H^0(S,E^\vee(N)\otimes I_\xi).
\]
Since $\det(E)=L$, it follows that $E^\vee(M)\cong E(-N)$,
and $E^\vee(N)\cong E(-M)$. Therefore, one has that
$h^0(S,E^\vee(N)\otimes I_\xi)=h^0(S,E(-M)\otimes I_\xi)$.
Using extension (\ref{eq: DM}), we claim  that
$h^0(S,M\otimes N^{\vee})=h^0(S,E(-N))$. Indeed,  if $\xi\ne\emptyset$, then $h^0(S,I_\xi)=0$.
If $\xi=\emptyset$,  the image of $1\in H^0(S, \OO_S)$ under the map $H^0(S,\OO_S)\to H^1(S,M\otimes N^{\vee})$ is
precisely the extension class, hence it is non-zero.

Observe that
$H^0(S,\OO_S)\cong H^0(S,E(-M))$, in particular,
$h^0(S,E^\vee(N) \otimes I_\xi)\le 1$. On the other hand,
the morphism
$
 H^0(S,E\otimes E^\vee)\to
H^0(S,E^\vee(N)\otimes I_\xi)
$
maps  $\mbox{id}_E$ to the arrow $E\to N\otimes I_\xi$, hence
it is non-zero. It follows that $h^0(S,E^\vee(N)\otimes I_\xi)=1$,
and moreover, the map $H^0(S,E\otimes E^\vee)\to
H^0(S,E^\vee(N)\otimes I_\xi)$ is surjective.
\endproof

In order to parameterize all pairs $(C, A)$
with non-simple  Lazarsfeld-Mukai bundles, we need a global construction.
We fix a non-trivial globally generated line bundle
$N$ on $S$ with $H^0(L(-2N))\neq 0$,
and an integer $\ell\ge 0$. We set $M:=L(-N)$
and $g:=1+L^2/2$.
Define $\widetilde{\P}_{N,\ell}$ to be the family of {\em vector bundles} of rank $2$ on
 $S$ given by non-trivial extensions
\begin{equation}
\label{eq: extension}
0\to M\to E\to N\otimes I_\xi\to 0,
\end{equation}
where $\xi$ is a zero-dimensional lci subscheme
of $S$ of length $\ell$,
and set
\[
\P_{N,\ell}:=\{[E]\in\widetilde{\P}_{N, \ell}:\ h^1(S,E)=h^2(S,E)=0\}.
\]
Equivalently (by Riemann-Roch), $[E]\in \P_{N, \ell}$ if and only if
$h^0(S,E)=g-c_2(E)+3$ and $h^1(S,E)=0$. Note that any non-simple Lazarsfeld-Mukai bundle on $S$ with
determinant $L$ belongs to some family $\P_{N,\ell}$.

\begin{rmk}
Using the Cayley-Bacharach property, we observe that
$\widetilde{\P}_{N,\ell}\ne \emptyset$
whenever $\Ext^1_S(N\otimes I_\xi,M)\ne 0$.
\end{rmk}

\begin{rmk}
\label{rmk: vanish N}
If $\P_{N,\ell}\ne\emptyset$, then $h^1(S,N)=0$ and
$h^0(S,N\otimes I_\xi)=h^0(S,N)-\ell$. Indeed, we choose
$[E]\in\P_{N,\ell}$. Then
\begin{eqnarray*}
h^0(S,E)& = &h^0(S,M)+h^0(S,N\otimes I_\xi)-h^1(S,M)\\
& \ge & h^0(S,M)+h^0(S,N)-\mathrm{length}(\xi)-h^1(S,M)\\
& \ge & \chi(S,M)+\chi(S,N)-\ell=2+\frac{1}{2}M^2+2+\frac{1}{2}N^2
-\ell\\
& = & 2+\frac{1}{2}L^2-M\cdot N +2-\ell
=g+3-c_2(E).
\end{eqnarray*}
Since $h^0(S,E)=g+3-c_2(E)$, all the inequalities are actually
equalities, hence
$h^1(S,N)=0$ and $h^0(S,N\otimes I_\xi)=h^0(S,N)-\ell$.
\end{rmk}

The family $\P_{N,\ell}$,
which, a priori, might be the empty set, is an open Zariski subset
of a projective bundle of the Hilbert scheme $S^{[\ell]}$, as
shown below:

\begin{lem}
\label{lemma: Ext}
If $\xi\in S^{[\ell]}$ and $\Ext^1_S(N\otimes I_\xi,M))\ne 0$, then
\[
\dim\ \Ext^1_S(N\otimes I_\xi,M)=\ell+h^1(S,M\otimes N^\vee)-h^2(S, M\otimes N^\vee).
\]
\end{lem}

\proof
Let $E$ be a vector bundle given by a non-trivial extension
\[
0\to M\to E\to N\otimes I_\xi\to 0.
\]
Applying $\Hom_S(\ -\ ,M)$ to this extension, we obtain the exact sequence
\[
H^0(S,\OO_S)\to \Ext^1_S(N\otimes I_\xi,M)\to
H^1(S,E^\vee(M))\to H^1(S,\OO_S)=0.
\]

Since $1\in H^0(S,\OO_S)$ is mapped to the
extension class of $E$ which is non-zero, it follows that
$
\dim \ \Ext^1_S(N\otimes I_\xi,M)=h^1(S,E^\vee(M))+1=h^1(S, E(-N))+1.
$
We apply the identification $E^\vee(M)\cong E(-N)$ as well as
the Riemann-Roch theorem for  $E(-N)$ and $M-N$ - note that
$c_1(E(-N))=M-N$ and $c_2(E(-N))=\ell$
(compute the Chern classes from the defining extension
twisted with $N^{\vee}$):
\[
\chi(S,E(-N))=4+\frac{1}{2}\bigl(M-N\bigr)^2-\ell
=2+\chi(S,M-N)-\ell.
\]

We note that $h^0(S,E(-N))=h^0(S,M-N)$. Indeed, if $\ell\ge 1$ then
$h^0(S,I_\xi)=0$, and if $\ell=0$ use that
$1\in H^0(S, \OO_S)$ is mapped to the extension class through
$H^0(S,\OO_S)\to H^1(S,M-N)$. Moreover,
$
h^2(S,E(-N))=h^0(S,E(-M))=1
$,
and we write that
$
\chi(S, E(-M))=2+h^0(S,M\otimes N^{\vee})-h^1(S,M\otimes N^{\vee})-\ell,
$
that is,
$$
h^1(S,E(-N))=\ell-1+h^1(S,M\otimes N^{\vee})-h^2(S, M\otimes N^{\vee}).
$$\endproof
Assuming that $\P_{N,\ell}\ne \emptyset$,
we consider  the  Grassmann bundle $\G_{N,\ell}$ over $\P_{N,\ell}$
classifying  pairs $(E,\Lambda)$ with $[E]\in\P_{N,\ell}$ and
$\Lambda\in \mathrm G(2,H^0(S,E))$. If $d:=c_2(E)$ we define the rational map $h_{N, \ell}: \G_{N, \ell}
\dashrightarrow \W^1_d(|L|)$, by setting $h_{N, \ell}(E, \Lambda):=(C_{\Lambda}, A_{\Lambda})$, where $A_{\Lambda}\in \mbox{Pic}^d(C_{\Lambda})$ is such that the following exact sequence on $S$ holds:
$$0\to \Lambda\otimes \OO_S\stackrel{\mathrm{ev}_{\Lambda}}\to E\to K_{C_{\Lambda}}\otimes A_{\Lambda}^{\vee}\to 0.$$

\begin{lem}
\label{lemma: Grass bundle}
If $\P_{N,\ell}\ne\emptyset$,
then \
$\dim\ \G_{N,\ell}=g+\ell+h^0(S,M\otimes N^{\vee})$.
\end{lem}

\proof Let $[E]\in\P_{N,\ell}$.
From Proposition \ref{lem: E_{C,A}} (ii), it is clear that
\[
\dim \ \G_{N,\ell} =2\ell+\dim\ \PP\bigl(\Ext^1_S(N\otimes I_\xi,M)\bigr)
+2(g+1-c_2(E)).
\]
Applying Lemma \ref{lemma: Ext}, as well as the fact that $l=c_2(E)-M\cdot N$, we find that
$$
\dim \ \G_{N,\ell}=2g-3M\cdot N+c_2(E)+1+h^1(S,M-N)-h^2(S, M-N)$$
$$=(g+c_2(E)-M\cdot N)+\bigl(g-2M\cdot N+1+h^1(S,M-N)-h^2(S, M-N)\bigr).
$$

From Riemann-Roch, we can write
\[
\chi(S,M-N)=2+\frac{1}{2}(M-N)^2=2+\frac{1}{2}L^2-2M\cdot N
=g+1-2M\cdot N.
\]
The conclusion follows.
\endproof

\begin{lem}
\label{lemma: W}
Assume that $\P_{N,\ell}$ contains a Lazarsfeld-Mukai vector bundle $E$ on $S$ with $c_2(E)=d$,
and let $\W\subset \W^1_d(|L|)$ be the closure of the
image of the rational map $h_{N,\ell}:\G_{N,\ell}\dashrightarrow \W^1_d(|L|)$.
Then $\dim\ \W =g+d-M\cdot N=g+\ell$.
\end{lem}

\proof
Clearly $\W$ is irreducible, as $\G_{N,\ell}$ is irreducible.
If $(C,A)\in \mbox{Im}(h_{N, \ell})$, then the fibre
$h^{-1}_{N, \ell}(C,A)$ is isomorphic to
the projectivization of the space of morphisms from
$E_{C,A}$ to $K_C\otimes A^{\vee}$. From (\ref{eqn: h^0(E F)}),
$\mathrm{Hom}(E_{C,A},K_C\otimes A^{\vee})$ is
isomorphic to $H^0(S,E_{C, A}\otimes F_{C, A})$, and has
dimension $h^0(S,M\otimes N^{\vee})+1$, because of Lemma \ref{lemma: endo}. Therefore, the
general fibre of $h_{N, \ell}$ has dimension $h^0(S,M\otimes N^{\vee})$. We apply
now Lemma \ref{lemma: Grass bundle}.
\endproof

\begin{lem}
\label{lem: M.N}
Suppose that a smooth curve $C\in|L|$ has Clifford
dimension one and $A$ is a globally generated line
bundle on $C$ with $h^0(C,A)=2$ and $[E_{C,A}]\in\P_{N,\ell}$.
Then $M\cdot N\ge \gon(C)$.
\end{lem}

\proof
By Lemma \ref{lemma: DM} it follows that $M|_C$ contributes to $\mbox{Cliff}(C)$.
From the exact sequence
$
0\to N^{\vee} \to M \to M|_C\to 0
$
and from the observation that $h^1(S,N)=0$ (see Remark
\ref{rmk: vanish N}),
we obtain by direct computation that
$$
\Cliff(M|_C)=M\cdot N+M^2-2h^0(S, M)+2=M\cdot N-2-2h^1(S,M)\ge k-2,
$$
that is, \   $M\cdot N\ge k+2h^1(S,M)\ge k$.
\endproof

\begin{rmk}
If we drop the condition on the Clifford dimension in the hypothesis
of Lemma \ref{lem: M.N}, we obtain the inequality $M\cdot N\ge \Cliff(C)+2$.
\end{rmk}

So far, we took care of indecomposable non-simple Lazarsfeld-Mukai
bundles, and computed the dimensions of the corresponding parameter spaces.
The decomposable case is much simpler. Let $E=E_{C,A}=M\oplus N$
be a decomposable Lazarsfeld-Mukai bundle.
It was proved in \cite{La1}
that the differential of the natural map
$h_E:{G}(2,H^0(S,E))\dashrightarrow |L|_s$
at a point $[\Lambda]$, with $\Lambda=H^0(C,A)^\vee$, coincides with the multiplication
map $\mu_{0,A}$.
Hence, if the Grassmannian ${G}(2,H^0(S,E))$ dominates the linear
system, the multiplication map is surjective at a general point
and the corresponding
irreducible components of the Brill-Noether loci are zero-dimensional.
This case can occur only if the Brill-Noether number is non-negative.

\medskip

All these intermediate results amount to the following:
\begin{thm}
\label{thm: Green Cliffdim 1}
Let $S$ be a $K3$ surface and $L$ a globally generated line bundle on $S$, such that general curves in $|L|$ are of Clifford dimension one. Suppose that
$\rho(g,1,k)\leq 0$, where $L^2=2g-2$ and $k$ is the (constant) gonality of
all curves in $|L|_s$. Then a
general curve $C\in|L|$ satisfies the linear growth condition (\ref{lgc}), thus Green's Conjecture is verified for {\em any} smooth curve in $|L|$.
\end{thm}

In the case $\rho(g,1,k)=1$, Green's Conjecture is also verified
for smooth curves in $|L|$, cf. \cite{Voisin: odd}, \cite{HR98}. To sum up,
Green's Conjecture is verified for any curve of Clifford
dimension one on a $K3$ surfaces.

\proof
It
suffices to estimate the dimension of dominating
irreducible components $\W$ of $\W^1_{k+n}(|L|)$, with $n=0, \ldots, g-k+2$,
with general point corresponding to a non-simple indecomposable Lazarsfeld-Mukai bundle.
Lemmas \ref{lemma: W} and \ref{lem: M.N} yield $\mbox{dim } \W\leq g+n$, which finishes the proof.
\endproof

\begin{rmk}
The proof of Theorem \ref{thm: Green Cliffdim 1} shows  that for $d> g-k+2$, every dominating component
of $\W^1_{d}(|L|)$ corresponds to simple Lazarsfeld-Mukai bundles.
In particular, for a general curve $C\in|L|$, one has
$\mathrm{dim }\ W^1_d(C)=\rho(g,1,d)$.
\end{rmk}

\begin{rmk}
The problem of deciding
whether Lazarsfeld-Mukai bundles appear in a given space $\P_{N,\ell}$
is a non-trivial one, cf. Remark \ref{rmk: LM open}.
\end{rmk}

\section{A criterion for the Green-Lazarsfeld Gonality Conjecture}
\label{sec: GL}

Along with Green's Conjecture, another statement of
similar flavor was proposed by Green and Lazarsfeld, \cite{GL86}.

\begin{con}{\rm (The Gonality Conjecture)}
\label{gon conj}
For any smooth curve $C$ of gonality $d$, every non-special globally
generated line bundle $L$ on $C$ of sufficiently high degree satisfies
$K_{h^0(L)-d,1}(C,L)=0$.
\end{con}

Conjecture \ref{gon conj} is equivalent to the seemingly weaker statement that \emph{there exists} a globally generated line bundle  $L\in \mbox{Pic}(C)$ with $h^1(C, L)=0$ for which the Koszul vanishing holds \cite{Ap1}. On a curve $C$ with the lgc property (\ref{lgc}), line bundles of
type $K_C(x+y)$, where $x, y\in C$ are general points,
verify the Gonality Conjecture \cite{Ap05}.
In particular, Theorem \ref{thm: Green Cliffdim 1} implies the following:

\begin{cor}
Let $S$ be a $K3$ surface and $L$ a globally generated line bundle on
$S$, such that general curves in $|L|$ are of Clifford dimension one. Then a
general curve $C\in|L|$ verifies Conjecture \ref{gon conj}.
\end{cor}

The main result of this short section is a refinement of the
main result of \cite{Ap05}:

\begin{thm}
\label{thm: GL}
Let $C$ be a smooth curve of Clifford dimension one and $x, y\in C$ be distinct points,
and denote
\[
\mathcal{Z}_n:=\{A\in W^1_{k+n}(C):\ h^0(C,A(-x-y))\ge 1\}.
\]
Suppose that
$\dim\ \mathcal{Z}_n\le n-1$,
for all $0\le n\le g-2k+2$.
Then the bundle $K_C(x+y)$ verifies the Gonality Conjecture.
\end{thm}

The condition in the statement of Theorem \ref{thm: GL}
means that passing through the points $x$ and $y$ is a non-trivial condition
on any irreducible component of maximal allowed dimension
$n$ of the Brill-Noether locus $W^1_{k+n}(C)$, for all $0\leq n\leq g-2k+2$.

\proof
The proof is an almost verbatim copy of the proof of \cite{Ap05} Theorem 2.
Define $\nu:=g-2k+2$. The idea is to show that for any $0\le n\le \nu$,
and for $(n+1)$ pairs of distinct general points
$x_0+y_0,x_1+y_1,\ldots, x_n+y_n\in C_2$,
there is no line bundle $A\in W^1_{k+n}(C)$ with
$h^0(C, A(-x_i-y_i))\ne 0$ for all $1\le i\le n$, such that either
$h^0(C, A(-x-y))\ne 0$ or $h^0(C, A(-x_0-y_0))\ne 0$. To this end, consider the
incidence varieties
\[
\left(\prod_{i=1}^nC_2\right)\times \mathcal{Z}_n\supset
\{(x_1+y_1,\ldots, x_n+y_n,A):\ h^0(A(-x_i-y_i))\ne 0,\ \forall i\},
\]
respectively,
\[
\left(\prod_{i=1}^{n+1}C_2\right)\times W^1_{k+n}(C)\supset
\{(x_0+y_0,x_1+y_1,\ldots, x_n+y_n,A):\ h^0(A(-x_i-y_i))\ne 0,\ \forall i\}.
\]
The fibres of the projection to $\mathcal{Z}_n$ are $n$-dimensional,
hence the incidence variety is at most $(2n-1)$-dimensional and
it cannot dominate $\prod_{i=1}^nC_2$. Similarly, the second variety
is at most $(2n+1)$-dimensional. Note that the condition to pass
through a pair of general points is a non-trivial condition on every
variety of complete pencils. To conclude, apply \cite{Ap05} Proposition 8.
\endproof

\section{Curves of higher Clifford dimension}
\label{sec: ELMS}

We analyze the Koszul cohomology of
curves of higher Clifford dimension
on a $K3$ surface $S$. This case has similarities to
 \cite{ApP08}, where one focused on $K3$ surfaces with Picard number $2$.
Since plane curves are known to verify Green's Conjecture,
the significant cases occur when the Clifford dimension is at least $3$.
Note that, unlike the Clifford index, the Clifford dimension
is {\em not} semi-continous. An example was given by Donagi-Morrison \cite{DM89}:
If $\epsilon: S\rightarrow \PP^2$ is a double sextic  and $L=\epsilon^*(\mathcal{O}_{\mathbb P^2}(3))$,
then the  general element in $|L|$ is isomorphic to a smooth plane sextic, hence it has Clifford dimension $2$,
while special points correspond to bielliptic curves and are of Clifford dimension $1$.

It was proved in \cite{CP95} and \cite{Kn09}
that, except for the Donagi-Morrison example, if a globally generated linear system $|L|$ on $S$ contains smooth curves
of Clifford dimension at least $2$, then $L=\OO_S(2D+\Gamma)$,
where $D, \Gamma\subset S$ are smooth
curves, $D^2\ge 2$ (hence $h^0(S,\OO_S(D))\ge 2$),
$\Gamma^2=-2$ and $D\cdot\Gamma=1$; the case when $L$ is ample is treated in \cite{CP95}, whereas the general case when $L$ is globally generated is settled in \cite{Kn09}. If the genus
of $D$ is $r\ge 3$, then the genus of a smooth curve
$C\in|L|$ equals $4r-2\ge 10$, and $\mbox{gon}(C)=2r$, while $\mbox{Cliff}(C)=2r-3$; the Clifford dimension of
$C$ is~$r$. From now on, we assume that we are in this situation.

Green's hyperplane section theorem implies that the Koszul cohomology
is constant in a linear system.
As in \cite{ApP08}, we degenerate a smooth curve $C\in |2D+\Gamma|$ to
a reducible curve $X+\Gamma$ with $X\in |2D|$.
In order to be able to carry out this plan, we first analyze
the geometry of the curves in $|2D|$. Notably, we shall prove:

\begin{thm}
\label{thm: hyp 2D}
The hypothesis of Theorem \ref{thm: GL} are verified for a general
curve $X\in |2D|$ and the two points of intersection  $X\cdot \Gamma$.
\end{thm}

The proof of Theorem \ref{thm: hyp 2D} proceeds in several steps.
The first result describes the fundamental invariants of a quadratic complete intersection section of $S$:

\begin{lem}
\label{lem: gon 2D}
Any smooth curve $X\in |2D|$ has genus $4r-3$, gonality
$2r-2$, and $\mathrm{Cliff}(X)=2r-4$.
\end{lem}

\proof
Since $\Cliff(D|_X)=2r-4$, we obtain $\Cliff(X)\le 2r-4<(g(X)-1)/2$,  that is,
$\mbox{Cliff}(X)$ is computed by a line bundle $B\in \mbox{Pic}(S)$, cf. \cite{GL87}.
Both bundles $B$ and $B^\prime:=\OO_S(X)\otimes B^\vee$ are globally generated,
hence $B\cdot \Gamma\ge 0$ and $B^\prime \cdot\Gamma\ge 0$. Since
$X\cdot \Gamma =2$, we
can assume that $B\cdot \Gamma\le 1$.
We may also assume, cf. \cite{Ma89} Corollary 2.3, that $h^0(S, B)=h^0(X,B|_{X})$ and $h^0(S, B')=h^0(X, B'_{| X})$. Then if $C\in |L|$ is smooth as above,  we obtain the estimate
\[
\Cliff(X)=B\cdot X-2h^0(S,B)+2\ge B\cdot C-2h^0(C,B|_C)+1\geq 2r-4.
\]
 Since $X$ has Clifford dimension $1$, it follows that $\mbox{gon}(X)=2r-2$.
\endproof

It suffices therefore to analyze the structure of the loci $W^1_{2r-2+n}(X)$ where $n\leq 3=g(X)-2\mathrm{gon}(X)+2$,
and more precisely those components of dimension $n$.

\begin{lem}
We fix a general $X\in |2D|$, viewed as a half-canonical curve $X\stackrel{|D|}\longrightarrow \PP^r$.
\begin{itemize}
\item $W^1_{2r-2}(X)$ is finite and all minimal pencils $\mathfrak{g}^1_{2r-2}$ on $X$ are given by the rulings of quadrics of rank $4$ in $H^0(\PP^r, \mathcal{I}_{X/\PP^r}(2))$.
\item $X$ has no base point free pencils $\mathfrak{g}^1_{2r-1}$, that is, $W^1_{2r-1}(X)=X+W^1_{2r-2}(X)$.
\item For $n=2, 3$, if $A\in W^1_{2r-2+n}(X)$ is a base point free pencil,
then the vector bundle $E_{X, A}$ is not simple.
\end{itemize}
In all cases $n\leq 3$, if $A$ belongs to an $n$-dimensional
component of $W^1_{2r-2+n}(X)$, then the corresponding DM extension
\[
0\to M\to E_{X,A}\to N\otimes I_\xi\to 0
\]
verifies $\mathrm{length}(\xi)=n$, $M\cdot N=2r-2$ and
$M\cdot \Gamma=N\cdot \Gamma =1$. When $n=2, 3$, we can take $M=N=\OO_X(D)$.
\end{lem}

\proof
We use Accola's lemma, cf. \cite{ELMS} Lemma 3.1. If $A\in W^1_{2r-2+n}(X)$ is base point free with $n\leq 3$, then
 $h^0(X,\OO_X(D)\otimes A^{\vee})\geq  2-n/2$. In particular, when $n=0, 1$, we find that $A':=\OO_X(D)\otimes A^{\vee}$ is a pencil as well. When $n=1$, we find that
 $A'\in W^1_{2r-3}(X)$, which is impossible, that is, $X$ carries no base point free pencils $\mathfrak g^1_{2r-1}$. If $n=0$, then $\mbox{deg}(A)=\mbox{deg}(A')=2r-2$ and this corresponds to a quadric $Q\in H^0(\PP^r, \I_{X/\PP^r}(2))$ with $\mbox{rk}(Q)=4$ and $X\cap \mbox{Sing}(Q)=\emptyset$, such that the  rulings of $Q$ cut out on $X$, precisely the pencils $A$ and $A'$ respectively. If $n=2, 3$, we find that  $h^0(X,K_X(-2A))\ne 0$, thus the kernel of the Petri map  $\mbox{Ker }\mu_{0,A}=H^0(X, K_X(-2A))\neq 0$, and then
the Lazarsfeld-Mukai bundle $E_{X,A}$ cannot be simple.

The vector bundle $E=E_{X, A}$ is thus expressible as a $DM$ extension
\[
0\to M\to E_{X,A}\to N\otimes I_\xi\to 0,
\]
and we recall that $N$ is globally generated with $h^1(S,N)=0$.
Suppose first that $n\ne 0$, thus $n\in \{2, 3\}$. Then $h^0(S,M\otimes N^\vee)\ne 0$, for otherwise
$\xi=\emptyset$, the extension is split, and the
split case only produces zero-dimensional components
of the Brill-Noether loci, whilst we are in the higher dimensional
case.
From the exact sequence defining $E_{X, A}$ coupled with Accola's Lemma, we obtain the isomorphisms $$H^0(S, E(-D))\cong H^0(X, K_X(-A)\otimes \OO_X(-D))=H^0(X, \OO_X(D-A))\neq 0,$$
therefore
$H^0(S, M(-D))\cong H^0(S, E(-D))\neq 0$.
Choose an effective divisor $F\in |M(-D)|$ and then $F\in |\OO_S(D)(-N)|$ as well.
From Lemmas \ref{lemma: W} and \ref{lem: M.N}
and the generality assumption on $X$, we find that $\mbox{length}(\xi)=n$, \  $M\cdot N=\gon(X)=2r-2$ and $h^1(S, M)=0$.
Furthermore, one computes that $F^2=0$. Since,
by degree reasons,
$h^0(S, \OO_S(F))=h^0(X, \OO_X(D-A))=1$
one obtains that $F\equiv 0$,
that is, $M=N=\OO_X(D)$.

If $n=0$, then in the associated DM extension, $\xi=0$, and
$M, N\in \mbox{Pic}(S)$ are globally generated,
$M\cdot N=2r-2$ and $h^0(M)=h^0(M|_X)$ and
$h^0(N)=h^0(N|_X)$.
The intersection of $\Gamma$
with one of the bundles $M$ or $N$ is $\le 1$; suppose $M\cdot \Gamma\le 1$.
We choose a smooth curve $C\in |2D+\Gamma|$, and compute
\[
\Cliff(M|_C)=M\cdot C-2h^0(M|_C)+2
\le M\cdot X-2h^0(M)+2+1=\Cliff(X)+1,
\]
hence $M$ computes $\Cliff(C)$ and
$M\cdot \Gamma=N\cdot \Gamma =1$.
\endproof

\begin{lem}
\label{lem: through x and y}
Let $X\in |2D|$ be any smooth curve, and $x,y\in  X\cdot \Gamma$. For any integer $n\ge 0$,
and any base point free pencil $A\in W^1_{2r-2+n}(X)$, the following are equivalent:
\begin{enumerate}
\item $h^0(X,A(-x-y))\ne 0$;
\item $E_{X,A}|_\Gamma\cong\mathcal{O}_\Gamma\oplus\mathcal{O}_\Gamma(2)$.
\end{enumerate}
\end{lem}

\begin{proof}
The non-vanishing of
$H^0(X,A(-x-y))$ is equivalent to \begin{equation}\label{fibers}
h^0(X,A(-x-y)) =1.
\end{equation}
Twisting the defining exact sequence of $E_{X,A}$ by $\mathcal{O}_S(\Gamma)$,
we obtain
\begin{equation}\label{El}
 0\to H^0(X,A)^\vee\otimes \mathcal{O}_S(\Gamma)
 \to E_{X,A}\otimes \mathcal{O}_S(\Gamma)\to K_X\otimes A^\vee(x+y)\to 0.
\end{equation}
By Riemann-Roch, $H^1 (S, \mathcal{O}_S(\Gamma))=0$. By taking cohomology,
$h^0(C,A(-x-y))= 1$ if and only if
$h^0 (S,E_{X,A}\otimes \mathcal{O}_S(\Gamma))= 2r+3-n$.
On the other hand,
$h^0(S,E_{X,A})= 2+h^1(X,A)= 2r+2-n$. Consider the (twisted) exact sequence defining $\Gamma$:
$$
 0\to E_{X,A}\to E_{X,A}\otimes\mathcal{O}_S (\Gamma) \to E_{X,A}|_{\Gamma} (-2) \to 0.
$$
We find by taking cohomology that  $x, y\in C$ lie in the same fibre of $|A|$ if and only if $h^0(\Gamma, E_{X,A| \Gamma}(-2))=1$.
Expressing  $E_{X,A}|_{\Gamma}= \mathcal{O}_{\Gamma} (a) \oplus \mathcal{O}_{\Gamma} (b)$,
with $a+b=2$, condition (\ref{fibers}) becomes equivalent to
$E_{X,A}|_\Gamma\cong\mathcal{O}_\Gamma\oplus\mathcal{O}_\Gamma(2)$.
\end{proof}

\noindent \emph{Proof of Theorem \ref{thm: hyp 2D}}.
For a base point free $A\in W^1_{2r-2+n}(X)$ with $n=0, 2, 3$, the vector bundle $E:=E_{X, A}$ appears as an
extension
\[
0\to M\to E\to N\otimes I_\xi\to 0,
\]
with $\mbox{length}(\xi)=n$, and
$M\cdot \Gamma=N\cdot \Gamma =1$. Recall that being
a Lazarsfeld-Mukai bundle is an open condition in any flat
family of bundles, Remark \ref{rmk: LM open}.
Hence, LM bundles in a parameter space $\mathcal{P}_{N,n}$ correspond to general cycles $\xi\in S^{[n]}$.
For $n=0$, we see immediately that $E|_\Gamma=\OO_\Gamma(1)^{\oplus 2}$.
For $n=2,3$ the same is true for $\xi$ such that $\xi\cap \Gamma=\emptyset$.
To conclude, we apply Lemma~\ref{lem: through x and y}.
\hfill$\Box$

\begin{rmk}
The variety $\W^1_{2r-2}(|2D|)$ is birationally equivalent to the parameter space of pairs
$(Q, \Pi)$, where $Q\in |\OO_{\PP^r}(2)|$ is a quadric of rank $4$ and $\Pi\subset Q$ is a ruling.
In particular, $\W^1_{2r-2}(|2D|)$ is irreducible (and of dimension $g$).
\end{rmk}

Theorem \ref{thm: hyp 2D} and Theorem \ref{thm: GL} imply the following:

\begin{cor}
\label{cor: GL 2D}
For a general curve $X\in|2D|$, we have
$K_{2r,1}(X,K_X\otimes \OO_S(\Gamma))=0$.
\end{cor}

The main result of this section is (compare to \cite{ApP08}):

\begin{thm}
\label{thm: Green ELMS}
Smooth curves of Clifford dimension at least three
on $K3$ surfaces satisfy Green's Conjecture.
\end{thm}

\proof
As in \cite[Section 4.1]{ApP08}, for all $p\ge 1$, we have  isomorphisms
$$K_{p,1}(X+\Gamma,\omega_{X+\Gamma})\cong K_{p,1}(X,K_X(\Gamma)).$$
Corollary \ref{cor: GL 2D} shows that
$K_{2r,1}(X+\Gamma,\omega_{X+\Gamma})=0$,
implying the vanishing of $K_{2r,1}(S,L)$, via Green's hyperplane
section theorem. Using the hyperplane section
theorem again, we obtain $K_{2r,1}(C,K_C)=0$,
for any smooth curve $C\in|L|$, that is, the vanishing predicted
by Green's Conjecture for $C$.
\endproof

Theorems \ref{thm: Green Cliffdim 1} and \ref{thm: Green ELMS},
 altogether complete
the proof of Theorem \ref{maintheorem}.

\end{document}